\documentclass[a4paper,12pt,reqno]{amsart}
\usepackage[english]{babel}
\usepackage[T1]{fontenc}
\usepackage[left=1.1in,right=1.1in,top=1.2in,bottom=1.2in]{geometry}
\usepackage{times}
\usepackage{microtype}

\usepackage{commath,amssymb,amscd,mathrsfs,mathtools,dsfont,bbm,multirow,array,caption}
\usepackage[all]{xy}
\usepackage{enumerate}
\usepackage{longtable}

\usepackage[usenames,dvipsnames,svgnames,table]{xcolor}
\definecolor{red}{RGB}{255,25,25}
\definecolor{blue}{RGB}{25,50,200}
\usepackage{tikz-cd}
\usetikzlibrary{decorations.pathmorphing}

\usepackage[pagebackref,linktocpage]{hyperref}

\hypersetup{
pdftitle={\@title},			% title
pdfauthor={\authors},		% author
colorlinks=true,				% false: boxed links; true: colored links
linkcolor=red,				% color of internal links (change box color with linkbordercolor)
citecolor=MidnightBlue,		% color of links to bibliography
filecolor=magenta,			% color of file links
urlcolor=MidnightBlue			% color of external links
}

\usepackage{cleveref} % after hyperref!!

\newtheorem{theorem}{Theorem}[section]
\crefname{theorem}{Theorem}{Theorems}
\newtheorem{lemma}[theorem]{Lemma}
\crefname{lemma}{Lemma}{Lemmas}

\crefname{proposition}{Proposition}{Propositions}

\crefname{prop}{Proposition}{Propositions}

\crefname{corollary}{Corollary}{Corollaries}

\crefname{cor}{Corollary}{Corollaries}

\crefname{conjecture}{Conjecture}{Conjectures}
\newtheorem{conj}[theorem]{Conjecture}
\crefname{conj}{Conjecture}{Conjectures}
\newtheorem*{conj*}{Conjecture}
\crefname{conj}{Conjecture}{Conjectures}

\theoremstyle{definition}

\crefname{definition}{Definition}{Definitions}
\newtheorem{defn}[theorem]{Definition}
\crefname{defn}{Definition}{Definitions}
\newtheorem{example}[theorem]{Example}
\crefname{example}{Example}{Examples}

\crefname{notation}{Notation}{Notation}
\newtheorem*{notation*}{Notation}
\crefname{notation}{Notation}{Notation}
\newtheorem{problem}[theorem]{Problem}
\crefname{problem}{Problem}{Problems}
\newtheorem{question}[theorem]{Question}
\crefname{question}{Question}{Questions}

\crefname{condition}{Condition}{Conditions}

\crefname{assumption}{Assumption}{Assumptions}

\theoremstyle{remark}
\newtheorem{rmk}[theorem]{Remark}
\crefname{rmk}{Remark}{Remarks}
\newtheorem*{rmk*}{Remark}
\crefname{rmk}{Remark}{Remarks}

\crefname{remark}{Remark}{Remarks}

\crefname{fact}{Fact}{Facts}

\crefname{claim}{Claim}{Claims}
\newtheorem*{claim*}{Claim}
\crefname{claim}{Claim}{Claims}

\crefname{step}{Step}{Steps}

\crefname{case}{Case}{Cases}

\numberwithin{equation}{section}
\newcommand{\bbC}{\mathbb{C}}

\newcommand{\bbQ}{\mathbb{Q}}
\newcommand{\bbR}{\mathbb{R}}

\newcommand{\bR}{\mathbf{R}}
\newcommand{\bZ}{\mathbf{Z}}

\newcommand{\NE}{\overline{\operatorname{NE}}}

\begin{document}

\title[Bounding cohomology]{Bounding cohomology on a smooth projective surface}

\author{Sichen Li}
\address{School of Mathematical Science, Shanghai Key Laboratory of PMMP, East China Normal University, Math. Bldg , No. 500, Dongchuan Road, Shanghai, 200241, P. R. China}
\email{\href{mailto:lisichen123@foxmail.com}{lisichen123@foxmail.com}}
\urladdr{\url{https://www.researchgate.net/profile/Sichen_Li4}}
\begin{abstract}
The following conjecture arose out of discussions between B. Harbourne, J. Ro\'e, C. Cilberto and R. Miranda: for a smooth projective surface $X$ there exists a positive constant $c_X$ such that $h^1(\mathcal O_X(C))\le c_X h^0(\mathcal O_X(C))$ for every prime divisor $C$ on $X$.
We show that the conjecture is true for  some smooth projective surfaces with Picard number 2.
\end{abstract}


%\date{\today}


\thanks{The Research was partially supported by the National Natural Science Foundation of China (Grant No. 11471116, 11531007), Science and Technology Commission of Shanghai Municipality (Grant No. 18dz2271000) and the China Scholar Council `High-level university graduate program'.}

\maketitle

%%%%%%%%%%%%%%%%%%%%%%%%%%%%%%%%%%%%%%%%%%%%%%%%%%%%%%%%%%

\section{Introduction}
In this note we work over the field $\bbC$ of complex numbers.
By a $(negative)~ curve$ on a surface we will mean a reduced, irreducible curve (with negative self-intersection).
By a $(-k)$-$curve$, we mean a negative curve $C$ with $C^2=-k<0$.

The bounded negativity conjecture (BNC for short)  is one of the most intriguing problems in the theory of projective surfaces and can be formulated as follows.
\begin{conj}\cite[Conjecture 1.1]{B.etc.13}\label{BNC}
For a smooth projective surface $X$ there exists an integer $b(X)\ge0$ such that $C^2\ge-b(X)$ for every curve $C\subseteq X$.
\end{conj}
Let us say that a smooth projective surface  $X$ has
\begin{equation*}
b(X)>0
\end{equation*}
 if there is at least one negative curve on $X$.

In \cite{BPS17}, T. Bauer, P. Pokora and D. Schmitz established the following theorem.
\begin{theorem}\cite[Theorem]{BPS17}\label{BPS17Thm}
For a smooth projective surface $X$ over an an algebraic closed field the following two statements are equivalent:
\begin{enumerate}
\item[(i)] $X$ has bounded Zariski denominators.
\item[(ii)] X  satisfies the BNC.
\end{enumerate}
\end{theorem}
Here, $X$ has bounded Zariski denominators (cf. \cite{BPS17}) if there exists an integer $d(X)\ge1$ such that for every pseudo-effective integral divisor $D$ the denominators in the Zariski decomposition of $D$ are bounded from above by $d(X)$  (cf. \cite{Zariski62, Fujita79}).

The main aim of this note is to study the following conjecture, which implies \cref{BNC} (cf. \cite[Proposition 14]{C.etc.17}).
\begin{conj}\cite[Conjecture 2.5.3]{B.etc.12}\label{BH}
Let $X$ be a smooth projective surface.
Then there exists a positive constant $c_X$ such that  $h^1(\mathcal O_X(C))\le c_Xh^0(\mathcal O_X(C))$ for every curve $C$ on $X$.
\end{conj}
 In \cite{B.etc.12}, the authors disproved   \cref{BH} by giving a counterexample of surface of general type (cf. \cite[Corollary 3.1.2]{B.etc.12}).
 However, they pointed out that it  could still be true that  \cref{BH} holds when restricted to $rational$ surfaces, in any characteristic.
 Indeed, the smooth projective rational surfaces with an effective anticanoncial divisor satisfy  \cref{BH} (cf. \cite[Proposition 3.1.3]{B.etc.12}).
In particular, if $X$ is the  blow-up of $\mathbb P^2$ at $n$ generic points and $c_X=0$, then   \cref{BH} for this $X$ is an equivalent version of the  SHGH conjecture as follows.
\begin{conj} \cite[Conjecture 2.5.1]{B.etc.12} \label{SHGH}
Let $X$ be the blow-up of $\mathbb P^2$ at $n$ generic points. Then $h^1(X,\mathcal O_X(C))=0$ for every curve $C$ on $X$.
\end{conj}
In order to give our main result, we now recall  the following question posed in \cite{C.etc.17}.
\begin{question}\cite[Question 4]{C.etc.17}\label{Que4}
Does there exist a constant $m(X)$ such that $\frac{(K_X\cdot D)}{D^2}<m(X)$ for any effective divisor $D$ with $D^2>0$ on a smooth projective surface  $X$?
\end{question}
 If  \cref{BH} is true for a smooth projective surface $X$, then $X$ is affirmative for  \cref{Que4} (cf. \cite[Proposition 15]{C.etc.17}).
This motivates us to give the following definition.
\begin{defn}
Let $X$ be a smooth projective surface.
\begin{enumerate}
\item[(1)] For every $\bbR$-divisor $D$ with $D^2\ne0$ on $X$, we define a value of $D$ as follows:
\begin{equation*}
                                                       l_D:=\frac{(K_X\cdot D)}{\max\bigg\{ 1, D^2\bigg\}}.
\end{equation*}
\item[(2)] For every $\bbR$-divisor  $D$ with $D^2=0$ on $X$, we define a value of $D$ as follows:
\begin{equation*}
                           l_D:=\frac{(K_X\cdot D)}{\max\bigg\{1,h^0(\mathcal O_X(D))\bigg\}}.
\end{equation*}
\item[(3)] $X$ satisfies $\mathbf{Hyp(A)}$ if  $\NE(X)=\sum_{i=1}^{\rho(X)}\bbR_{\ge0}[C_i]$ such that each $C_i$ is a curve. Here, $\rho(X)$ is the Picard number of $X$.
\item[(4)] $X$ satisfies $\mathbf{Hyp(B)}$ if there exists a positive constant $m(X)$ such that $l_C\le m(X)$ for every curve $C^2\ne0$ on $X$.
\item[(5)] $X$ satisfies $\mathbf{Hyp(C)}$ if there exists a positive constant $m(X)$ such that $l_C\le m(X)$ for every curve $C$ on $X$.
\end{enumerate}
\end{defn}
To solve   \cref{BH} partially, for the case when  $\rho(X)=2$, we give the main result as follows.
\begin{theorem}\label{Main}
Let $X$ be a smooth projective surface. The following statements hold.
\begin{enumerate}
\item[(1)]  If $X$ satisfies the BNC and there exists a positive constant $m(X)$ such that $l_C\le m(X)$ for every curve $C$ on $X$ and  $D^2\le m(X)h^0(\mathcal O_X(D))$ for every curve $D$ with $l_D>1$ and $D^2>0$ on $X$, then $X$ satisfies  \cref{BH}.
\item[(2)] Suppose $\kappa(X)=0$ and the canonical divisor $K_X$ is nef. Then $X$ satisfies   \cref{BH}.
\item[(3)] Suppose  $\rho(X)=2$ and $\kappa(X)=-\infty$. Then $X$ satisfies $\mathbf{Hyp(B)}$. In particular, every ruled surface with  invariant $e>0$ satisfies \cref{BH}.
\item[(4)]  Suppose $\rho(X)=2$ and $X$ has two negative curves. Then $X$ satisfies \cref{BH}.
\item[(5)] Supoose $\rho(X)=2, \kappa(X)=1$ and $b(X)>0$. Then $X$ satisfies \cref{BH}.
\item[(6)] Suppose $\rho(X)=2$ and $X$ satisfies  $\mathbf{Hyp(A)}$. Then $X$ satisfies $\mathbf{Hyp(B)}$.
\end{enumerate}
\end{theorem}
\begin{rmk}
\begin{enumerate}
\item It is hard to establish that there exists a positive constant $m(X)$ such that $l_D\le m(X)$ for $D\in|nC|$ with the Iitaka dimension $\kappa(X,C)=1$ and $n\gg0$, where $C$ is a curve on $X$.
This is related to effectivity of Iitaka fibrations, which are known for the pluricanonical system $|mK_X|$ of every smooth projective variety $X$ in arbitrary dimension (cf. \cite{Iitaka70,TX09,BZ16}).
Therefore, we have to consider a weaker hypothesis $\mathbf{Hyp(B)}$.
\item In \cite[Claim 2.11]{Li19},  we give a classification of the smooth projective surfaces $X$ with $\rho(X)=2$ and two negative curves $C_1$ and $C_2$.
Here, the closed Mori cone $\NE(X)=\bbR_{\ge0}[C_1]+\bbR_{\ge0}[C_2]$, i.e, $X$ satisfies $\mathbf{Hyp(A)}$.
Moreover, see \cref{Hyp(A)E} about \cref{Main}(6).
\end{enumerate}
\end{rmk}
\section{The Proof  of   \cref{Main}}
In this section, we divide our proof of   \cref{Main} into some steps.
\begin{proof}[Proof~of~\cref{Main} (1)]
Take a curve $C$ on $X$. Note that by Serre duality (cf. \cite[Corollary III.7.7 and III.7.12]{Hartshorne77}), $h^2(\mathcal O_X(C))=h^0(\mathcal O_X(K_X-C))\le p_g(X)$.
As a result,
\begin{equation}\label{eq1}
h^2(\mathcal O_X(C))-\chi(\mathcal O_X)\le  q(X)-1.
\end{equation}
Here, $p_g(X)$ and $q(X)$ are the geometric genus of $X$ and the irregularity of $X$ respectively. Our main condition is the following:

(*) \label{*} There exists a positive constant $m(X)$ such that $l_C\le m(X)$ for every curve $C$ on $X$ and  $D^2\le m(X)h^0(\mathcal O_X(D))$ for every curve $D$ with $l_D>1$ and $D^2>0$ on $X$.

We divide the proof into the following three cases.

Case~(i).
Suppose $C^2>0$. Then by Riemann-Roch theorem (cf. \cite[Theorem V.1.6]{Hartshorne77}),
\begin{equation}\label{eq2}
                                             h^1(\mathcal O_X(C))=h^0(\mathcal O_X(C))+h^2(\mathcal O_X(C))-\chi(\mathcal O_X)+\frac{C^2(l_C-1)}{2}.
\end{equation}
If $l_C\le1$, then Equation (\ref{eq1}) and (\ref{eq2}) imply  that $h^1(\mathcal O_X(C))\le h^0(\mathcal O_X(C))+q(X)-1$, which is the desired result by $c_X:=q(X)$.
If $l_C>1$, then Equation (\ref{eq1}) and (\ref{eq2})  and the condition (*)  imply that $2h^1(\mathcal O_X(C))\le (m^2(X)-m(X)+2)h^0(\mathcal O_X(C))+2(q(X)-1)$,
which is the desired result by $2c_X:=m^2(X)-m(X)+q(X)$.

Case~(ii).
Suppose $C^2=0$. Then by Riemann-Roch theorem,
\begin{equation}\label{eq3}
                                             2h^1(\mathcal O_X(C))=2h^2(\mathcal O_X(C))-2\chi(\mathcal O_X)+h^0(\mathcal O_X(C))(l_C+2),
\end{equation}
which, Equation (\ref{eq1}) and the condition (*) imply that
\begin{equation*}
                                             2h^1(\mathcal O_X(C))\le 2(q(X)-1)+h^0(\mathcal O_X(C))(m(X)+2),
\end{equation*}
which is the desired result by $2c_X:=m(X)+2q(X)$.

Case~(iii).
Suppose $C^2<0$. Then $h^0(\mathcal O_X(C))=1$. Since $X$ satisfies the BNC, there exists a positive constant $b(X)$ such that every curve $C$ on $X$ has $C^2\ge-b(X)$. By Riemann-Roch theorem,
\begin{equation}\label{eq4}
                                             2h^1(\mathcal O_X(C))=2+2h^2(\mathcal O_X(C))-2\chi(\mathcal O_X)+l_C-C^2,
\end{equation}
which, Equation (\ref{eq1}) and  the condition (*)  imply that $2h^1(\mathcal O_X(C))\le 2q(X)+m(X)+b(X)$, which is the desired result by $2c_X:=2q(X)+m(X)+b(X)$.

In all, we  complete the proof of \cref{Main}(1).
\end{proof}
\begin{rmk}\label{Rem1}
\begin{enumerate}
\item[(1)] Suppose $X$ satisfies   \cref{BH}.
Then by \cite[Proposition 15]{C.etc.17}, Equation (\ref{eq3}) and (\ref{eq4}), $X$ satisfies $\mathbf{Hyp(B)}$ and $\mathbf{Hyp(C)}$.
\item[(2)] The condition of \cref{Main}(1) may be not necessary.
Let $c_X\gg1$.
Take a curve $C$ with $C^2>0$ on $X$.
Suppose $X$ satisfies \cref{BH}.
Then Equation (\ref{eq1}) and (\ref{eq2}) imply that
\begin{equation*}\begin{split}
                                           \frac{C^2(l_C-1)}{2}&=h^1(\mathcal O_X(C))-h^0(\mathcal O_X(C))-h^2(\mathcal O_X(C))+\chi(\mathcal O_X)\\&\le (c_X-1)h^0(\mathcal O_X(C))+\chi(\mathcal O_X).
\end{split}\end{equation*}
If we can find a sequence $\big\{l_{C_i}\big\}$ on $X$ such that $C_i^2>0, l_{C_i}>1$ and $\lim_{i\to\infty} l_{C_i}=1$, then it is unknown that there exists a positive constant $m(X)$ such that $C^2\le m(X)h^0(\mathcal O_X(C))$ for every curve $C$ with $C^2>0$ and $l_C>1$ on $X$. Therefore, the following question is asked.
\end{enumerate}
\end{rmk}
\begin{question}
Let $X$ be a smooth projective surface. Suppose $X$ satisfies \cref{BH}. Is there  a positive constant $m(X)$ such that $C^2\le m(X)h^0(\mathcal O_X(C))$ for every curve $C$ with $C^2>0$ and $l_C>1$ on $X$?
\end{question}
\begin{proof}[$Proof~of~ \cref{Main} (2)$]
Since $\kappa(X)=0$ and $K_X$ is nef, $K_X\equiv0$ (numerical).
As a result, $l_C=0$ for every curve $C$ on $X$. By the adjunction formula, $C^2\ge-2$.
By Riemann-Roch theorem, $2h^1(\mathcal O_X(C))=2h^0(\mathcal O_X(C))+2h^2(\mathcal O_X(C))-2\chi(\mathcal O_X)-C^2$,
which and Equation (\ref{eq1}) imply that $h^1(\mathcal O_X(C))\le (q(X)+1)h^0(\mathcal O_X(C))$.
Therefore, $X$ satisfies  \cref{BH}.
\end{proof}
\begin{lemma}\label{Ruled}
Every ruled surface satisfies $\mathbf{Hyp(B)}$. In particular, every ruled surface with either invariant $e>0$ or $e=0$ over a curve of genus $g\le1$  satisfy Conjecture \cref{BH}.
\end{lemma}
\begin{proof}
Let $\pi: X\rightarrow B$ be a ruled surface over a smooth curve $B$ of genus $g$, with invariant $e$. Let $C\subseteq X$ be a section, and let $f$ be a fibre. By \cite[Proposition V.2.3 and 2.9]{Hartshorne77},
\begin{equation*}
                            \mathrm{Pic}~X\cong \mathbb Z C\oplus \pi^*\mathrm{Pic} B, C\cdot f=1, f^2=0, C^2=-e, K_X\equiv -2C+(2g-2-e)f.
\end{equation*}
Let $D\equiv aC+bf$ with $a,b\in\mathbb Z$ be a curve on $X$.
Now we divide the remaining proof  into the following four cases.

Case 1. Suppose $e>0$. Then by \cite[Proposition V.2.20(a)]{Hartshorne77},  $a>0, b\ge ae$.
As a result, every curve $D(\ne C,f)$ has $D^2>0$, $(C\cdot D)\ge0$ and $(f\cdot D)>0$. Thus,
\begin{equation*}\begin{split}
                                         l_D&=\frac{(K_X\cdot D)}{D^2}\\&\le\frac{2(C\cdot D)+|2g-2-e|(f\cdot D)}{a(C\cdot D)+b(f\cdot D)}\\&\le\mathrm{max}\bigg\{\frac{2}{a},\frac{|2g-2-e|}{b}\bigg\}.
\end{split}\end{equation*}
Here, $a$ and $b$ are positive integers.
Therefore, $X$ satisfies $\mathbf{Hyp(B)}$ and $\mathbf{Hyp(C)}$.
If $b\ge 2g-2-e$, then
\begin{equation}\label{eq5}
                        (K_X-D)D=-(2+a)(C\cdot D)+(2g-2-e-b)(f\cdot D)\le0.
\end{equation}
By Riemann-Roch theorem, Equation (\ref{eq1}) and (\ref{eq5}) imply that
\begin{equation*}\begin{split}
            h^1(\mathcal O_X(D))&=h^0(\mathcal O_X(D))+h^2(\mathcal O_X(D))+\frac{(K_X-D)D}{2}-\chi(\mathcal O_X)\\&\le q(X)h^0(\mathcal O_X(D)).
\end{split}\end{equation*}
If $b<2g-2-e$, then $a<(2g-2-e)e^{-1}$ by $b\ge ae$. As a result, $D^2<2(2g-2-e)^2e^{-1}$. Hence, by  \cref{Main}(i), $X$ satisfies  \cref{BH}.

Case 2.
Suppose $e=0$ and $g\le1$. Then by \cite[Proposition V.2.20(a)]{Hartshorne77}, $a>0$ and $b\ge0$.
As a result, $(K_X-D)D=-(2+a)b+a(2g-2-b)\le0$, which and Equation (\ref{eq1}) imply that
\begin{equation*}\begin{split}
            h^1(\mathcal O_X(D))&=h^0(\mathcal O_X(D))+h^2(\mathcal O_X(D))+\frac{(K_X-D)D}{2}-\chi(\mathcal O_X)\\&\le q(X)h^0(\mathcal O_X(D)).
\end{split}\end{equation*}
Therefore, $X$ satisfies  \cref{BH}.
In particular, by   \cref{Rem1}(1), $X$ satisfies $\mathbf{Hyp(B)}$.

Case~3.
Suppose $e=0$ and $g\ge2$.
Then by \cite[Proposition V.2.20(a)]{Hartshorne77}, $a>0$ and $b\ge0$.
As a result, every curve $D(\ne aC,f)$ has $D^2>0$, $(C\cdot D)>0$ and $(f\cdot D)>0$.
Note that every curve $D$ has zero self-intersection if and only if either $D\equiv aC$ or $D\equiv f$.
Suppose $b>0$.
Then
\begin{equation*}\begin{split}
                                         l_D&=\frac{(K_X\cdot D)}{D^2}\\&\le\frac{2(C\cdot D)+|2g-2|(f\cdot D)}{a(C\cdot D)+b(f\cdot D)}\\&\le\mathrm{max}\bigg\{\frac{2}{a},\frac{|2g-2|}{b}\bigg\}.
\end{split}\end{equation*}
Here, $a$ and $b$ are positive integers.
 Therefore, $X$ satisfies $\mathbf{Hyp(B)}$.

Case 4.
Suppose $e<0$.
Then by \cite[Proposition V.2.21]{Hartshorne77}, every curve $D$ has either $D^2=0$ or $D^2>0$.
Moreover, $D^2>0$  implies that $D$ is ample and $a>0, b>\frac{1}{2}ae$.
Now suppose $D^2>0$.
Then $D\cdot C>0$ and $D\cdot f>0$. Take $C'=C+\frac{1}{2}ef$ and then $D\cdot C'>0$ and
\begin{equation*}\begin{split}
                                         l_D&=\frac{(K_X\cdot D)}{D^2}\\&=\frac{-2(C\cdot D)+(2g-2-e)(f\cdot D)}{a(C'\cdot D)+(b-\frac{1}{2}ae)(f\cdot D)}\\&\le \frac{|4g-4-2e|}{2b-ae}.
\end{split}\end{equation*}
Here, $2b-ae$ is a positive integer. Therefore, $X$ satisfies $\mathbf{Hyp(B)}$.

In all, we complete the proof of  \cref{Ruled}.
\end{proof}
It is well-known that the smooth projective surfaces satisfy the minimal model conjecture (cf. \cite{KM98, BCHM10}) as follows.
\begin{lemma}\label{mmp}
Let $X$ be a smooth projective surface. If the canonical divisor $K_X$ is pseudo-effective, then the Kodaira dimension $\kappa(X)\ge0$.
\end{lemma}
\begin{lemma}\label{kappa<0}
Let $X$ be a smooth projective surface with $\rho(X)=2$.
 If $\kappa(X)=-\infty$, then $X$ satisfies  $\mathbf{Hyp(B)}$.
 In particular,  every ruled surface with either invariant $e>0$ or $e=0$ over a curve of genus $g\le1$ and one point blow-up of $\mathbb P^2$ satisfy  \cref{BH}.
\end{lemma}
\begin{proof}
Let $S$ be a relatively minimal model of $X$.
A smooth projective surface $S$ is relatively minimal if it has no $(-1)$-rational curves.
By the Enrique-Kodaira classification of relatively minimal surfaces (cf. \cite{Hartshorne77, BHPV04, KM98}), it must be one of the following cases: a surface with nef canonical divisor, a ruled surface or $\mathbb P^2$.
Since $\kappa(X)=-\infty$, by   \cref{mmp}, $K_X$ is not nef.
Therefore, $S$ is either a ruled surface or $\mathbb P^2$.
As a result, $\rho(X)=2$ implies that $X$ is either a ruled surface or one point blow-up of $\mathbb P^2$.
By  \cref{Ruled}, every ruled surface satisfies $\mathbf{Hyp(B)}$.
In particular, every ruled surface with invariant $e>0$ and every ruled surface with $e=0$ over a curve of genus $g\le1$  satisfy  \cref{BH}.
Now suppose $\pi: X \rightarrow\mathbb P^2$ is one point blow-up of $\mathbb P^2$ with a exceptional curve $E$ and $\mathrm{Pic}(\mathbb P^2)=\bZ[H]$, where $H=\mathcal O_{\mathbb P^2}(1)$.
 Then  $K_X=\pi^{*}(-3H)+E$ and $C=\pi^*(dH)-mE$, where $m:=\mathrm{mult}_p(\pi_*C)$ and $C$ is a curve on $X$.
 Note that $d\ge m$ since $\pi_*C$ is a plane projective curve.
 Thus, every curve $C$ (not $E$) on $X$ has $C^2\ge0$ and then $C$ is nef.
 Since $-K_X$ is ample, $C-K_X$ is ample.
 Therefore, by Kadaira vanishing theorem, $h^1(\mathcal O_X(C))=0$. Therefore, $X$ satisfies \cref{BH}. By  \cref{Rem1}(1), $X$ satisfies $\mathbf{Hyp(B)}$.
\end{proof}
\begin{lemma}\label{>0}
Let $X$ be a smooth projective surface with $\rho(X)=2$. Then the following statements hold.
\begin{enumerate}
\item[(i)] $\NE(X)=\bbR_{\ge0}[f_1]+\bbR_{\ge0}[f_2]$, $f_1^2\le0, f_2^2\le0$ and $f_1\cdot f_2>0$. Here, $f_1, f_2$ are extremal rays.
\item[(ii)] If a curve $C$ has $C^2\le0$, then $C\equiv af_1$ or $C\equiv bf_2$ for some $a,b\in\mathbb R_{>0}$.
\item[(iii)] Suppose a divisor $D\equiv a_1f_1+a_2f_2$ with $a_1,a_2>0$ in (i).  Then $D$ is big. Moreover, if $D$ is a curve, then $D$ is nef and big and $D^2>0$.
\end{enumerate}
\end{lemma}
\begin{proof}
By \cite[Lemma 1.22]{KM98}, (i) and (ii) are clear since $\rho(X)=2$. For (iii), $D\equiv a_1f_1+a_2f_2$ with $a_1, a_2>0$ is an interior point of Mori cone, then by \cite[Theorem 2.2.26]{Lazarsfeld04}, $D$ is big. Moreover, if $D$ is a curve, then $D$ is nef. As a result, $D^2>0$.
\end{proof}
\begin{lemma}\label{nef}
Let $X$ be a smooth projective surface with $\rho(X)=2$. If $X$ has two negative curves $C_1$ and $C_2$, then the nef cone $\mathrm{Nef}(X)$ is
\begin{equation*}
 \mathrm{Nef}(X)=\bigg\{ a_1C_1+a_2C_2\bigg| a_1(C_1\cdot C_2)\ge a_2(-C_2^2), a_2(C_1\cdot C_2)\ge a_1(-C_1^2), a_1>0, a_2>0\bigg\}.
\end{equation*}
\end{lemma}
\begin{proof}
Since $\rho(X)=2$, $\NE(X)=\bbR_{\ge0}[C_1]+\bbR_{\ge0}[C_2]$  by \cref{>0}(ii).
As a result, an effective $\bR$- divisor $D\equiv a_1C_1+a_2C_2$ is nef if and only if $D\cdot C_1\ge0$ and $D\cdot C_2\ge0$, which imply the desired result.
\end{proof}
\begin{lemma}\label{twoneg}
Let $X$ be a smooth projective surface with $\rho(X)=2$. Suppose $X$ has two negative curves $C_1$ and $C_2$. Then $X$ satisfies  \cref{BH}.
\end{lemma}
\begin{proof}
Note that $\NE(X)=\bbR_{\ge0}[C_1]+\bbR_{\ge0}[C_2]$ by \cref{>0}(ii).
We first show that $X$ satisfies $\mathbf{Hyp(B)}$.
By \cite[Claim 2.11]{Li19}, $\kappa(X)\ge0$, i.e., there exists a positive integral number $m$ such that $h^0(X, \mathcal O_X(mK_X))\ge0$.
Therefore, $K_X$ is $\mathbb Q$-effective  divisor.
As a result, $K_X\equiv  aC_1+bC_2$ with $a, b\in\bbR_{\ge0}$.
Take a curve $D\equiv a_1C_1+a_2C_2$ with $a_1, a_2>0$, then by  \cref{>0}(iii), $D^2>0$.
As a result, $D\cdot C\ge0$ and $X$ has no any curves with zero self-intersection.
$D^2\ge1$ implies that either $D\cdot C_1\ge1$ and $D\cdot C_2\ge0$ or $D\cdot C_1\ge0$ and $D\cdot C_2\ge1$.
Without loss of generality, suppose that $D\cdot C_2\ge0$ and $D\cdot C_1\ge1$.
Then $a_1\ge (C_1^2+(C_1\cdot C_2)^2(-C_2^2)^{-1})^{-1}$. Here,  $C_1^2+(C_1\cdot C_2)^2(-C_2^2)^{-1}>0$ since $\rho(X)=2$. By symmetry and \cref{nef},
\begin{equation*}
                                               a_i\ge c:=\min\bigg\{(C_i^2+\frac{(C_1\cdot C_2)^2}{-C_j^2})^{-1}, \frac{-C_j^2}{(C_1\cdot C_2)}(C_i^2+\frac{(C_1\cdot C_2)^2}{-C_j^2})^{-1}\bigg\},
\end{equation*}
where $i\ne j\in\{1,2\}$. Therefore,
\begin{equation*}\begin{split}
                                                    l_D&=\frac{a(D\cdot C_1)+b(D\cdot C_2)}{a_1(D\cdot C_1)+a_2(D\cdot C_2)}\\&\le\max\bigg\{\frac{a}{c}, \frac{b}{c}\bigg\}.
\end{split}\end{equation*}
So $X$ satisfies $\mathbf{Hyp(B)}$.
 If $a_1>a$ and  $a_2>b$, then
 \begin{equation*}
 (K_X-D)D=(a-a_1)(D\cdot C_1)+(b-a_2)(D\cdot C_2)<0.
 \end{equation*}
This and Equation (\ref{eq1}) imply that
\begin{equation*}\begin{split}
                                             h^1(\mathcal O_X(D))&=h^0(\mathcal O_X(D))+h^2(\mathcal O_X(D))+\frac{(K_X\cdot D)-D^2}{2}-\chi(\mathcal O_X)\\&\le q(X)h^0(\mathcal O_X(D)).
\end{split}\end{equation*}
If $a_1\le a$ or  $a_2\le b$, then by \cref{nef}, $a_2\le a(C_1\cdot C_2)(-C_2^2)^{-1}$ or $a_1\le b(C_1\cdot C_2)(-C_1^2)^{-1}$.
As a result,
\begin{equation*}
D^2\le \max\bigg\{2a^2(C_1\cdot C_2)^2(-C_2^2)^{-1},2b^2(C_1\cdot C_2)^2(-C_1^2)^{-1}\bigg\}.
\end{equation*}
Therefore, $X$ satisfies \cref{BH}  by \cref{Main}(1).
\end{proof}
\begin{lemma}\label{kappa(X)=1}
Let $X$ be a smooth projective surface with $\rho(X)=2$.
If $\kappa(X)=1$ and $b(X)>0$, then $X$ satisfies  \cref{BH}.
\end{lemma}
\begin{proof}
Since $\kappa(X)=1, \rho(X)=2$ and $\kappa(X)$ is a birational invariant, $K_X$ is nef and semi-ample.
By \cite[Proposition IX.2]{Beauville96}, we have $K_X^2=0$ and there is a surjective morphism $p: X\to B$ over a smooth curve $B$, whose general fibre $F$ is an elliptic curve.
Note that $X$ has exactly one negative curve $C$ by $b(X)>0$ and \cite[Claim 2.14]{Li19}.
In fact, $p$ is an Iitaka fibration of $X$.
In \cite{Iitaka70}, S. Iitaka proved that if $m$ is any natural number divisible by 12 and $m\ge86$, then $|mK_X|$ defines the Iitaka fibration.
Hence, there exists a curve $F$ as a general fiber of $p$ such that $F\equiv mK_X$.
Then by \cref{>0}(i)(ii), $\NE(X)=\bbR_{\ge0}[F]+\bbR_{\ge0}[C]$.
Note that $(F\cdot C)>0$ since $\rho(X)=2$.
Take a curve $D\equiv a_1F+a_2C$ with $a_1,a_2\ge0$.
 By  \cref{>0}(iii), $D^2>0$ if and only if $a_1,a_2>0$, $D^2=0$ if and only if $D\equiv a_1F$.
 Now suppose $D\equiv a_1F$. Then $l_D=0$.
 Note that $h^1(\mathcal O_X(D))\le q(X)h^0(\mathcal O_X(D))$ by Riemann-Roch theorem and Equation (\ref{eq1}).
 Now suppose $D^2>0$. Then $(F\cdot D)\ge1$ and $(C\cdot D)\ge0$, which imply that
\begin{equation}\label{eq6}
                                                 a_2\ge(F\cdot C)^{-1}, a_1\ge a_2(-C^2)(F\cdot C)^{-1}.
\end{equation}
Therefore, by Equation (\ref{eq6}),
\begin{equation*}\begin{split}
                        l_D&=\frac{(F\cdot D)}{m(a_1(F\cdot D)+a_2(C\cdot D))}\\&\le (F\cdot C)^2(-mC^2)^{-1}.
\end{split}\end{equation*}
Hence, $X$ satisfies $\mathbf{Hyp(C)}$.
If $ma_1\ge1$, then  $(K_X-D)D=(1-ma_1)(K_X\cdot D)-a_2(C\cdot D)\le0$.
As a result, $h^1(\mathcal O_X(D))=q(X)h^0(\mathcal O_X(C))$ by Riemann-Roch theorem and Equation (\ref{eq1}).
If $ma_1<1$, then by Equation (\ref{eq6}), $a_2<(F\cdot C)(-mC^2)^{-1}$. So $D^2<2m^{-2}(F\cdot C)^2(-C^2)^{-1}$. Hence, by   \cref{Main}(1), $X$ satisfies  \cref{BH}.
\end{proof}
\begin{lemma}\label{Hyp(A)}
Let $X$ be a smooth projective surface with $\rho(X)=2$.
Suppose $X$ satisfies  $\mathbf{Hyp(A)}$.
Then $X$ satisfies $\mathbf{Hyp(B)}$.
Moreover, if  $\NE(X)$ is generated by two curves with zero Iitaka dimension, then $X$ satisfies $\mathbf{Hyp(C)}$.
\end{lemma}
\begin{proof}
 By \cref{kappa<0}, we can assume that $\kappa(X)\ge0$, i.e., there exists a positive integral number $m$ such that $h^0(X, \mathcal O_X(mK_X))\ge0$.
Therefore, $K_X$ is $\bbQ$-effective  divisor.
Since $X$ satisfies $\mathbf{Hyp(A)}$, by   \cref{>0}(ii)(i), $\NE(X)=\bbR_{\ge0}[C_1]+\bbR_{\ge0}[C_2]$. Here, $C_1$ and $C_2$ are two curves and $C_1^2, C_2^2\le0$.
As a result, $K_X\equiv aC_1+bC_2$ with $a,b\ge0$.

If $C_1^2<0$ and $C_2^2<0$, it follows from \cref{twoneg} and \cref{Rem1}(1).

Now suppose $C_1^2=0$. Then $X$ has at most one negative curve.
 By  \cref{>0}(iii),  $D^2>0$ if and only if $D\equiv a_1C_1+a_2C_2$ with $a_1, a_2>0$, $D^2>0$ if and only if $D$ is nef and big. As a result, $D^2>0$ implies that  $(D\cdot C_1)\ge1$, i.e., $a_2\ge(C_1\cdot C_2)^{-1}$. Now we divide the remaining proof into the following two cases.

Case (i).
 Suppose $C_2^2=0$. Then $X$ has no negative curve. If $D^2>0$, then $(D\cdot C_2)\ge1$, which implies that $a_1\ge(C_1\cdot C_2)^{-1}$. Therefore,
\begin{equation*}\begin{split}
                      l_D&=\frac{a(D\cdot C_1)+b(D\cdot C_2)}{a_1(D\cdot C_1)+a_2(D\cdot C_2)}\\&\le\mathrm{max}\bigg\{a(C_1\cdot C_2), b(C_1\cdot C_2)\bigg\}.
\end{split}\end{equation*}
Therefore, $X$ satisfies $\mathbf{Hyp(B)}$. Moreover, if $\kappa(X, C_1)=\kappa(X, C_2)=0$, then  $X$ has only two curves with zero self-intersection. Hence, $X$ satisfies $\mathbf{Hyp(C)}$.

Case~(ii).
Suppose $C_2^2<0$. $X$ has only one negative curve $C_2$. If $D^2>0$, then $(D\cdot C_2)\ge0$, which  implies that $a_1\ge a_2(-C_2^2)(C_1\cdot C_2)^{-1}$. Therefore,
\begin{equation*}\begin{split}
                      l_D&=\frac{a(D\cdot C_1)+b(D\cdot C_2)}{a_1(D\cdot C_1)+a_2(D\cdot C_2)}\\&\le\mathrm{max}\bigg\{a(-C_2^2)^{-1}(C_1\cdot C_2)^2, b(C_1\cdot C_2)\bigg\}.
\end{split}\end{equation*}
Therefore, $X$ satisfies $\mathbf{Hyp(B)}$. Moreover, if $\kappa(X, C_1)=0$, then $X$ has only one curve with zero self-intersection. Hence, $X$ satisfies $\mathbf{Hyp(C)}$.
\end{proof}
\begin{rmk}\label{Hyp(A)E}
For the examples of   \cref{Hyp(A)}, there exists a K3 surface with two negative curves (cf. \cite[Theorem 2]{Kovacs94}, \cite[Claim 2.12]{Li19}), \cref{kappa(X)=1} is an example of  $X$ with $\kappa(X, C_1)=1$ and $\kappa(X, C_2)=0$. The remaining case is that $\kappa(X)\ge0, C_1^2=0$ and $\kappa(X,C_1)=\kappa(X,C_2)=0$.
J. Ro\'e told us that there is an example for the case as follows.
\end{rmk}
\begin{example}\label{Roe}
(J. Ro\'e's example)
Let $\pi: X\to Y$ be one-point blow-up of a smooth projective surface $Y$ with $\rho(Y)=1$.
Then one extremal ray in $\overline{NE}(X)$ is the exceptional curve $E=C_2$.
Since $\rho(Y)=1$, the other extremal ray is determined by the Seshadri constant (cf. \cite[Definition 5.1.1]{Lazarsfeld04}).
Let $A$ be the ample generator of  $\mathrm{Pic}~(Y)$.
Because of the duality between $\mathrm{Nef}(X)$ and $\NE(X)$ (cf. \cite[Proposition 1.4.28]{Lazarsfeld04}), the class of $[C_1]$ is $\pi^*A-mE$, where $m=\frac{A^2}{e}$ and $e$ is the Seshadri constant of $A$ at the point that is blown-up.
The assumption $C_1^2=0$ implies that $m=e=\sqrt{A^2}$.
This does happen for the general smooth surfaces of degree $e^2$ in $\mathbb P^3$ (cf. \cite{Steffens98, Bauer97}).
The easiest case would be a general point $p$ on a general quartic surface $Y\subseteq\mathbb P^3$.
Blow it up. $A$ is the pullback to $Y$ of the hyperplane class in $\mathbb P^3$.
$C_1$ has class $\pi^*A-2E$, which is given by the strict transform of the nodal quartic curve obtained as intersection of the quartic surface with its tangent plane at $p$.
So $C_1$ is a plane nodal quartic, and it has genus $g\ge2$.
The restriction of $n(\pi^*A-2E)$ to $C_1$ has degree zero.
To prove that $\kappa(X,C_1)=0$, we need to check that $nC_1$ is the only section of $n(\pi^*A-2E)$. Note that the degree zero intersection divisor is not torsion in $\mathrm{Pic}~(C_1)$, which is true if $p$ and $Y$ are general. Assume that another section $D$ such that $D\sim_{\bbQ}nC_1+T$, where effective $\bbQ$-divisor $T$ is algebraically equivalent to zero. As a result, $T|_{C_1}=0$, a contradiction.
\end{example}
\begin{proof}[Proof~of~\cref{Main}(3)$\sim$(6)]
(3) follows from  \cref{kappa<0}.
 (4) follows from  \cref{twoneg}.
 (5) follows from  \cref{kappa(X)=1}.
 (6) follows from \cref{Hyp(A)}.
\end{proof}
We end by posing the following problem.
\begin{problem}
Classify all algebraic surfaces with $\mathbf{Hyp(B)}$.
\end{problem}
\begin{rmk}
For every smooth projective surface $X$, we conjecture that $X$ satisfies  $\mathbf{Hyp(B)}$ if $X$ satisfies  $\mathbf{Hyp(A)}$.
\end{rmk}
\section*{Acknowledgments}
He thanks Rong Du, Piotr Pokora, Joaquim Ro\'e, Hao Sun, De-Qi Zhang, Lei Zhang and Guolei Zhong for valuable conversations and the anonymous referee for several suggestions.

%%%%%%%%%%%%%%%%%%%%%%%%%%%%%%%%%%%%%%%%%%%%%%%%%%%%%%%%%%

%\linespread{1.1}

%\bibliographystyle{amsplain}
\bibliographystyle{amsalpha}
\bibliography{../mybib}

\providecommand{\bysame}{\leavevmode\hbox to3em{\hrulefill}\thinspace}
\providecommand{\MR}{\relax\ifhmode\unskip\space\fi MR }
\providecommand{\MRhref}[2]{%
  \href{http://www.ams.org/mathscinet-getitem?mr=#1}{#2}
}
\providecommand{\href}[2]{#2}
\begin{thebibliography}{BCHM10}

\bibitem{Bauer97}
T. Bauer,
\emph{Seshadri constants of quartic surfaces,}
Math. Ann. \textbf{309}(1997), 475-481.

\bibitem{B.etc.12}
T. Bauer, C. Bocci, S. Cooper, S. D. Rocci, M. Dumnicki, B. Harbourne, K. Jabbusch, A. L. Knutsen,
A. K$\mathfrak{\ddot{u}}$ronya, R. Miranda, J. Ro$\mathrm{\acute{e}}$, H. Schenck, T. Szemberg, and Z. Teithler,
\emph{Recent developments and open problems in linear series,}
In: Contributions to Algebraic Geometry, p. 93-140, EMS Ser. Congr. Rep., Eur. Math. Soc., Z\"{u}rich, 2012.

\bibitem{B.etc.13}
T. Bauer, B. Harbourne, A. L. Knutsen, A. K$\mathrm{\ddot{u}}$ronya, S. M$\mathrm{\ddot{u}}$ller-Stach, X. Roulleau, and T. Szemberg,
\emph{Negative curves on algebraic surfaces,}
Duke Math. J. \textbf{162}(10)(2013), 1877-1894.

\bibitem{Beauville96}
A. Beauville,
\emph{Complex algebraic surfaces,}
2ed., London Mathematical Society Student Texts, \textbf{34}, Cambridge University Press, Cambridge, 1996.

 \bibitem{BCHM10}
C. Birkar, P. Cascini, C. D. Hacon, and J. McKernan,
  \emph{Existence of minimal models for varieties of log general type},
  J. Amer. Math. Soc. \textbf{23} (2010), no.~2, 405--468.

\bibitem{BHPV04}
W. P. Barth, K. Hulek, C. A. M. Peters, and A. Van De Ven,
 \emph{Compact Complex Surfaces,} Ergeb. Math. Grenzgeb. Springer-Verlag, Berlin, 2004.

\bibitem{BPS17}
T. Bauer, P. Pokora, and D. Schmitz,
\emph{On the boundedness of the denominators in the Zariski decomposition on surfaces,} J. Reine Angew. Math. \textbf{733}(2017),251-259.

\bibitem{BZ16}
C. Birkar and D. Q. Zhang,
\emph{Effectively of Iitaka fibrations and pluricanonical systems of polarized pairs,}
Publ. Math. Inst. Hautes \'{E}tudes Sci. \textbf{123} (2016), 283-331.

\bibitem{C.etc.17}
C. Ciliberto, A. L. Knutsen, J. Lesieutre, V. Lozovanu, R. Miranda, Y. Mustopa, and D. Testa,
\emph{A few questions about curves on surfaces,} Rend. Circ. Mat. Palermo. II. Ser. \textbf{66} (2)(2017),195-204.

\bibitem{Fujita79}
T. Fujita,
\emph{On Zariski problem,}
Proc. Japan Acad. Ser. A \textbf{55}(1979), 106-110.

\bibitem{Hartshorne77}
R. Hartshorne,
\emph{Algebraic Geometry,} GTM \textbf{52}, Springer-Verlag, New York, 1977.

\bibitem{Iitaka70}
S. Iitaka,
\emph{Deformations of compact complex surfaces, II,}
J. Math. Soc. Soc. Japab, \textup{22}(1970), 247-261.

\bibitem{KM98}
J. Koll$\mathrm{\acute{a}}$r and S. Mori,
\emph{Birational geometry of algebraic varieties,}
 Cambridge Tracts in Mathematics, \textbf{134}, Cambridge University Press, Cambridge, 1998.

 \bibitem{Kovacs94}
 S. J. Kov$\mathrm{\acute{a}}$cs,
 \emph{The cone of curves of a K3 surface,}
 Math. Ann. \textup{300} (4)(1994),681-691.

\bibitem{Lazarsfeld04}
R. Lazarsfeld, \emph{Positivity in algebraic geometry, I}, Ergebnisse der
  Mathematik und ihrer Grenzgebiete. \textbf{48}, Springer-Verlag, Berlin, 2004.

\bibitem{Li19}
S. Li,
\emph{A note on a smooth projective surface with Picard number 2,}
Math. Nachr. \textbf{292} (2019), no. 12, 2637-2642.

 \bibitem{Steffens98}
 A. Steffens,
 \emph{Remarks on Seshadri constants,}
 Math. Z. \textbf{227}(1998),505-510.

 \bibitem{TX09}
 G. Todorov and C. Xu,
 \emph{Effectiveness of the log Iitaka fibration for 3-folds and 4-folds,}
 Algebra   Number Theory, \textbf{3} (2009), 697-710.

 \bibitem{Zariski62}
 O. Zariski,
 \emph{The Theory of Riemann-Roch for high multiples of an effective divisor on an algebraic surface,}
 Ann. Math. \textbf{76}(1962),560-615.

\end{thebibliography}

\end{document}